\documentclass{amsart}
\usepackage{amsthm,amssymb}
\usepackage{stmaryrd,tikz,pgfplots}
\usepackage{marginnote}

\usepackage[textsize=tiny]{todonotes}

\newcommand{\Marginpar}[1]{\marginpar{\tiny{#1}}}
\newcommand{\Note}[1]{{\par\noindent\hrulefill\par\tiny{#1}\par\noindent\hrulefill\par}}
\newcommand{\Detail}[1]{{#1}}
\renewcommand{\Marginpar}[1]{}
\renewcommand{\Note}[1]{}
\renewcommand{\Detail}[1]{}
\newcommand\ddfrac[2]{\frac{\displaystyle #1}{\displaystyle #2}}

\usepackage{hyperref}
\usepackage[all]{xy}
\usepackage{amsmath,amsfonts}	
\usepackage{amsthm,latexsym,amssymb}
\usepackage{tensor}

\newtheorem{thm}{Theorem}[section]
\newtheorem{prop}[thm]{Proposition}
\newtheorem{lem}[thm]{Lemma}
\newtheorem{cor}[thm]{Corollary}

\theoremstyle{definition}
\newtheorem{defn}[thm]{Definition}

\newtheorem{rem}[thm]{Remark}

\renewcommand{\[}{\begin{equation*}}
\renewcommand{\]}{\end{equation*}}

\begin{document}
\parskip1mm

\title[Chern--Yamabe problem]{On the Chern--Yamabe flow}

\author{Mehdi Lejmi}
\address{Department of Mathematics, Bronx Community College of CUNY, Bronx, NY 10453, USA.}
\email{mehdi.lejmi@bcc.cuny.edu}

\author{Ali Maalaoui}
\address{Department of mathematics and natural sciences, American University of Ras Al Khaimah, PO Box 10021, Ras Al Khaimah, UAE.}
\email{ali.maalaoui@aurak.ac.ae}

\begin{abstract}
On a closed balanced manifold, we show that if the Chern scalar curvature
is small enough in a certain Sobolev norm then a slightly modified version of the Chern--Yamabe flow~\cite{Angella:2015aa} converges to a solution of the Chern--Yamabe problem.
We also prove that if the Chern scalar curvature, on closed almost-Hermitian manifolds, is close enough to a constant function in a H{\"o}lder norm then the Chern--Yamabe problem has a solution
for generic values of the fundamental constant.
\end{abstract}
\maketitle


\section{Introduction}

An almost-Hermitian manifold $(M,J,g)$ is equipped with a pair of an almost-complex structure $J$
and a Riemannian metric $g$ such that $g(\cdot,\cdot)=g(J\cdot,J\cdot).$ If $J$ is integrable then $(J,g)$ is a Hermitian structure.
On an almost-Hermitian manifold $(M,J,g)$, there exists a natural connection called the Chern connection~\cite{MR0066733,MR0165458,MR1456265}. From it one can derive the Chern scalar curvature.
In~\cite{Angella:2015aa}, Angella, Calamai and Spotti initiated the study of an analogue of the Yamabe problem on closed Hermitian manifolds. Namely, they
examined the existence of constant Chern scalar curvature metrics in a conformal class. The problem is then extended to the almost-Hermitian case in~\cite{Lejmi:2017aa} and it turns out that the conformal change
of the Chern scalar curvature is the same as in the integrable case.
The existence of metrics of zero Chern scalar curvature on symplectic Calabi--Yau $4$-manifolds is among the motivations to study such problem~\cite[Question 7.16]{MR2743450}.
For example, on the Kodaira--Thurston manifold, which is a symplectic Calabi--Yau $4$-manifold, there is no metric of zero Riemannian scalar curvature~\cite{MR1306021,MR1306023}
(see also~\cite{MR2743450}), even though it is well-known from the solution to the Yamabe problem~\cite{MR0125546,MR0240748,MR1636569,MR788292} that there exists a metric of constant Riemannian scalar
curvature in any conformal class. 
Nevertheless, there exists on the Kodaira--Thurston manifold metrics with vanishing Chern scalar curvature~\cite{MR2988734} (see also~\cite{MR2795448,MR2917134}).

Now, it turns out that on a closed almost-Hermitian manifold $(M,J,g)$, if the fundamental constant $C(M,J,[g])$~\cite{MR1712115,MR779217,MR0486672,MR0425191} (see Definition~\ref{constant})
is negative or zero then there exists a metric in the conformal class of $g$ with Chern scalar curvature equal to the fundamental constant~\cite{MR0301680,MR779217,Angella:2015aa,Lejmi:2017aa}.
In the negative case, first Berger~\cite{MR0301680} used a variational approach to prescribe the Chern scalar curvature in the conformal class of $g$ when $(M,J,g)$ is a closed K\"ahler manifold.
Angella, Calamai and Spotti~\cite{Angella:2015aa} used the continuity method to solve the Chern--Yamabe problem when $C(M,J,[g])<0$ on any Hermitian manifold.
The case $C(M,J,[g])=0$ is quite straightforward to solve using the existence of Gauduchon metrics~\cite[Corollary 1.9]{MR779217} (see also~~\cite{Angella:2015aa,Lejmi:2017aa}).
As expected, the positive case is more complicated because the PDE looses its nice analytic properties.
If $M$ is a $2$-dimensional compact manifold, the Chern scalar curvature is the Gaussian curvature. The problem is then variational and well-defined on the sobolev space $H^{1,2}(M)$ and the PDE appearing as the Euler-Lagrange equation of the energy functional is well-understood in that case (see for instance~\cite{MR681859,MR0295261,MR2001443,MR2483132} and the references therein).
In higher dimension, a flow approach is suggested in~\cite{Angella:2015aa} to tackle the problem.

In the present paper, after the preliminaries in~\S2, we prove in~\S3 that if the Chern scalar curvature, on a closed almost-Hermitian manifold, is close enough to a constant function in
the H{\"o}lder norm $C^{0,\alpha}(M)$ then
the Chern--Yamabe problem has a solution for generic values of the fundamental constant.
\begin{thm}\label{thm_3}
Let $(M,J,g)$ be a closed almost-Hermitian manifold of dimension $2n$ and $s^C$ its Chern scalar curvature.
Assume that ${C(M,J,[g])} \notin\sigma(\Delta+g(\theta,d\cdot))$, where $\sigma(\Delta + g(\theta,d\cdot))$ is the spectrum of the operator $\Delta+ g(\theta,d\cdot)$ (here
$\Delta$ stands for the Riemannian Laplacian of $g$
and $\theta$ for the Lee form of $(J,g)$).
Then there exists $\epsilon_{0}>0$ such that if $$\|s^C-C(M,J,[g])\|_{C^{0,\alpha}}<\epsilon_{0}$$ then there exists a conformal metric $\tilde{g}=e^{\frac{2u}{n}}g$ of constant Chern scalar curvature $C(M,J,[g])$, where the function $u$ is normalized by $\int_Me^{\frac{2u}{n}}\,\operatorname{vol}_g=1$.
\end{thm}
In~\S4, we first study the Chern--Yamabe flow defined in~\cite{Angella:2015aa},
when the fundamental constant is negative. We prove that the flow converges to a solution of the Chern--Yamabe problem (see Theorem~\ref{thm_1}) and so we recover Angella--Calamai--Spotti's result in the negative case.
Then,
we restrict ourselves to a closed balanced manifold $(M,J,g)$ of dimension $2n$ and we consider a slightly modified version of the Chern--Yamabe flow, namely
\begin{equation}\label{flow_mod}
\left\{\begin{array}{ll}
\ddfrac{\partial u}{\partial t}&=-\Delta u -s^C+\lambda(t)e^{\frac{2u}{n}},\\
\\
u(x,0)&=0,
\end{array}
\right.
\end{equation}
where $\lambda(t)=\frac{\int_M s^C\operatorname{vol}_g}{\int_{M}e^{\frac{2u}{n}}\operatorname{vol}_g}$.
We prove then that 
if the Chern scalar curvature is small enough in the Sobolev norm $H^{k,2}(M)$ (with $k>n$) then the flow~(\ref{flow_mod}) converges to a solution of the Chern--Yamabe problem.
\begin{thm}\label{thm_2}
Let $(M,J,g)$ be a closed balanced manifold of dimension $2n$, $s^C$ its Chern scalar curvature and $k>n$.
There exists $\delta>0$ such that if $\|s^C\|_{H^{k,2}}<\delta$ then the solution $u(x,t)$ of the flow $(\ref{flow_mod})$ converges to a solution of the the equation $$\Delta u+s^C=\frac{\int_M s^C\operatorname{vol}_g}{\int_{M}e^{\frac{2u}{n}}\operatorname{vol}_g}e^{\frac{2u}{n}},$$
in the $H^{k,2}$ sense.
\end{thm}
Finally in~\S5, we examine the case of a $2n$-dimensional closed almost-Hermitian manifold $(M,J,g)$ equipped with a free action of $(2n-2)$-dimensional subgroup $G$
of the isometry group preserving the pair $(J,g)$. We show then the existence of a $G$-invariant metric of constant Chern scalar curvature in the conformal class of $g$ (see Proposition~\ref{prop_1}).

\subsection*{Acknowledgments}
The first named author is supported in part by a PSC-CUNY Award $\#$ 60053-00 48, jointly funded by The Professional Staff Congress and The City University of New York.

\section{Preliminaries}\label{sec:Preliminaries}

An almost-Hermitian structure on a real manifold $M$ of dimension $2n$ is given by a pair $(J,g)$ of an almost-complex structure $J$
and a Riemannian metric $g$ satisfying
\[
g(\cdot,\cdot)=g(J\cdot,J\cdot).
\]
The almost-Hermitian structure induces the fundamental $2$-form $F(\cdot,\cdot):=g(J\cdot,\cdot).$ In general, the $2$-form $F$ is not closed and we have
\[
dF=\left(dF\right)_0+\frac{1}{n-1}\theta\wedge F,
\]
where $d$ is the exterior derivative, $\left(dF\right)_0$ is the trace-free part of $dF$ and $\theta$ is a $1$-form called the Lee form.
The metric $g$ is Gauduchon if $\delta^g\theta=0,$ where $\delta^g$ is the codifferential defined as the adjoint of
the exterior derivative with the respect to the global inner product induced by the metric $g.$ Gauduchon proved in~\cite{MR0470920}
that any confomal class $[g]=\{e^{\frac{2u}{n}}g\,|\, u\colon M\longrightarrow \mathbb{R}\}$ contains a unique Gauduchon metric up to a constant multiple.
If $dF=0$ then the almost-Hermitian structure is called almost-K\"ahler. More generally, if $\theta=0$, then the metric $g$ is said to be balanced.  

On the other hand, if the almost-complex structure $J$ is integrable then $(J,g)$ is actually a Hermtian structure and $(M,J)$ is a complex manifold.
A Hermitian structure  $(J,g)$ is K\"ahler if $dF=0.$

On an almost-Hermitian manifold $(M,J,g)$, the almost-complex structure $J$ is parallel with respect to the Levi-Civita connection $D^g$ if
and only if $(J,g)$ is K\"ahler. We consider then the Chern connection $\nabla$~\cite{MR0066733,MR0165458,MR1456265} defined as the unique connection satisfying $\nabla J=\nabla g=0$ and $T^\nabla_{JX,JY}=-T^{\nabla}_{X,Y},$ 
where $T^\nabla_{X,Y}=\nabla_XY-\nabla_YX-[X,Y]$ stands for the torsion of $\nabla$ and $X,Y$ are vector fields on $M$ (for more details see~\cite{MR1456265}).
We denote by $R^\nabla_{X,Y}=[\nabla_X,\nabla_Y]-\nabla_{[X,Y]}$ the curvature tensor of $\nabla$ and by $\rho_{X,Y}=-\Lambda\left(R^\nabla_{X,Y}\right)$
the first (or Hermitian) Ricci form, defined as the trace of $R^\nabla$ (here $\Lambda$ stands for the contraction by the fundamental form $F$).
The $2$-form $\rho$ is actually a representative of the first Chern class $2\pi c_1(TM,J)$ of $M.$ 
The Chern scalar curvature
 \[
 s^C=\Lambda\left(\rho\right)
 \]
 is defined as the trace of $\rho.$
 
 \subsection{Conformal variation}
 On the $2n$-dimensional almost-Hermitian manifold $(M,J,g)$, we consider a conformal metric $\tilde{g}=e^{\frac{2u}{n}}g$, where $u$ is a smooth real-valued function on $M$. 
 Then $(J,\tilde{g})$ is an almost-Hermitian structure. We denote by $s^C$, respectively $\tilde{s}^C$,
 the Chern scalar curvature of $(J,g)$, respectively $(J,\tilde{g})$. We have then the following formula for the conformal change \cite{MR742896,Lejmi:2017aa}
 \begin{equation}\label{conf_change}
\Delta u+g(\theta,du)+s^C= \tilde{s}^Ce^{\frac{2u}{n}},
 \end{equation}
 where $\Delta$ is the Riemannian Laplacian with respect to the metric $g$ and $\theta$ is the Lee form of $(J,g)$.
 
 \begin{defn}\label{constant}
 Let $(M,J,g)$ be a $2n$-dimensional closed almost-Hermitian manifold and $g_0$ be the unique Gauduchon metric in $[g]$ with total volume equals to $1$.
 Then the fundamental constant~\cite{MR1712115,MR779217,MR0486672,MR0425191} is (up to a factor 2)
 \[
 C(M,J,[g])=\int_Ms_0^C\frac{F_0^n}{n!},
 \]
 \end{defn}
 where $s_0^C$ is the Chern scalar curvature of $(J,g_0)$ and $F_0$ is the fundamental form induced by $(J,g_0).$
 We recall the the following observation of~\cite[Corollary 1.9]{MR779217} (see also~\cite{Angella:2015aa,Lejmi:2017aa})
\begin{prop}~\cite{MR779217}\label{sign_fixed}
 Let $(M,J,g)$ be a closed almost-Hermitian manifold. Then there exists a conformal metric $\tilde{g}\in[g]$
 whose Chern scalar curvature has the same sign as $C(M,J,[g])$ at every point.
 \end{prop}

\subsection{Chern--Yamabe problem}
In~\cite{Angella:2015aa}, Angella, Calamai and Spotti initiated the study of the Chern--Yamabe problem on closed Hermitian manifolds.
Namely, they investigated the existence of metrics of constant Chern scalar curvature in a given conformal class.
The problem was generalized to the non-integrable case in~\cite{Lejmi:2017aa}. It turns out that the conformal change of the Chern scalar curvature is the same
as in the integrable case and it is given by Equation~(\ref{conf_change}).

When $C(M,J,[g])=0$, Balas~\cite[Corollary 1.9]{MR779217} proved the existence of a flat Chern scalar curvature metric in the conformal class $[g].$
When $C(M,J,[g])< 0$, Angella, Calamai and Spotti~\cite{Angella:2015aa} showed the existence of negative constant Chern scalar
curvature metric in $[g]$ using the continuity method. We also refer to the work of Berger~\cite{MR0301680} when the conformal $[g]$ contains a K\"ahler metric (or more generally a balanced metric).
As expected, the case $C(M,J,[g])> 0$ is the problematic one. On complex manifolds $(M,J)$, $C(M,J,[g])> 0$
implies by the Gauduchon plurigenera Theorem~\cite{MR0486672} that the Kodaira dimension is $-\infty.$
In~\cite{Angella:2015aa}, Angella, Calamai and Spotti gave examples
of positive constant Chern scalar curvature Hermitian non K\"ahler metrics. For instance, 
they deformed flat Chern scalar curvature metrics, using implicit function theorem, to obtain families of positive constant Chern scalar curvature metrics.
They also suggested a Chern--Yamable flow to attack the problem.
Finally, the first named author and Upmeier~\cite{Lejmi:2017aa} studied the existence of metrics of constant Chern scalar curvature
on some ruled manifolds.
 
\section{Small Oscillation case}
In this section, we prove that if the Chern scalar curvature is close enough to a constant function in the H\"older norm $C^{0,\alpha}(M)$ then
the Chern-Yamabe problem is solvable for generic values of the fundamental constant. On a closed Riemannian manifold $(M,g)$, we consider first the following equation 
\begin{equation}\label{eq1}
\Delta u+S=\tilde{S}e^{\frac{2u}{n}}, 
\end{equation}
where $\Delta$ is the Riemannian Laplacian with respect to $g$ and $S$ and $\tilde{S}$ are given functions in $C^{0,\alpha}(M)$.

\begin{lem}
There exists $\epsilon_{0}>0$ and a $G_{\delta}$-dense set $$\mathcal{A}\subset  B_{\epsilon_{0}}(S)=\{u\in C^{0,\alpha}(M)\,|\,\|S-u\|_{C^{0,\alpha}}<\varepsilon_0\},$$ such that if $\tilde{S}\in \mathcal{A}$, Equation (\ref{eq1}) has at least one solution.
\end{lem}
\begin{proof}
First, we know that for a generic $u\in C^{0,\alpha}(M)$ we have that $\Delta-\frac{2u}{n}$ is invertible. So we asume then that $\tilde{S}$ is a such generic function in $C^{0,\alpha}(M)$ and  that $$\|S-\tilde{S}\|_{C^{0,\alpha}}<\epsilon_{0}.$$
We consider the set $$B_{\varepsilon}=\{u\in C^{0,\alpha}(M)\,|\,\|u\|_{C^{0,\alpha}}<\varepsilon\}.$$ 
We can rewrite Equation (\ref{eq1}) as
$$\Delta u-\frac{2}{n}\tilde{S}u=\tilde{S}-S+\tilde{S}(e^{\frac{2u}{n}}-1-\frac{2u}{n}).$$
Notice that if $\|u\|_{C^{0,\alpha}}<\varepsilon$ then 
$$\|\tilde{S}(e^{\frac{2u}{n}}-1-\frac{2u}{n})\|_{C^{0,\alpha}}<C\varepsilon^{2}.$$
Now, since $\Delta -\frac{2}{n}\tilde{S}$ is invertible, we consider the operator
$$T(u)=(\Delta-\frac{2}{n}\tilde{S})^{-1}(\tilde{S}-S+\tilde{S}(e^{\frac{2u}{n}}-1-\frac{2u}{n}))$$
Clearly if $|S-\tilde{S}|_{C^{0,\alpha}}$ is small, then by Schauder's estimates we have 
\begin{equation}\label{stab}
T(B_{\varepsilon})\subset B_{\varepsilon}.
\end{equation}
Moreover,
$$\|T(u)-T(v)\|_{C^{0,\alpha}}=\|O(u^{2})-O(v^{2})\|_{C^{0,\alpha}}\leq C\varepsilon \|u-v\|_{C^{0,\alpha}}$$ 
So if $\varepsilon$ is small enough, we do have a contraction and hence a fixed point thus a solution to $(\ref{eq1})$.
\end{proof}
\begin{cor}\label{cor_1}
Let $(M,J,g,F)$ be a closed balanced manifold of dimension $2n$ and $s^C$ its Chern scalar curvature.
Assume that ${C(M,J,[g])} \notin\sigma(\Delta)$, where $\sigma(\Delta)$ is the spectrum of $\Delta$. Then there exists $\epsilon_{0}>0$ such that if $$\|s^C-C(M,J,[g])\|_{C^{0,\alpha}}<\epsilon_{0}$$ then there exists a conformal metric $\tilde{g}=e^{\frac{2u}{n}}g$ of constant Chern scalar curvature $C(M,J,[g])$, where the function $u$ is normalized by $\int_Me^{\frac{2u}{n}}\,\frac{F^n}{n!}=1$.
\end{cor}
\begin{rem}\label{rem_1}
Notice that this can be extended in full generality to the case of the operator $\Delta + g(\theta,d\cdot)$ with the condition $C(M,J,[g])\not \in \sigma(\Delta + g(\theta,d\cdot))$. We deduce then Theorem~\ref{thm_3}.\\
Also, the $C^{0,\alpha}$ pinching can be weakened to an integral pinching condition. Indeed, the only place where the condition $\|S-\tilde{S}\|_{C^{0,\alpha}}<\epsilon_{0}$ is used was in $(\ref{stab})$. Notice then that if we assume instead that $\|S-\tilde{S}\|_{L^{p}}<\epsilon_{1}$, we have that for $p>\frac{2n}{2}=n$, the elliptic regularity of the operator $\Delta-\frac{2}{n}\tilde{S}$, combined with the Sobolev embedding, ensures that
$$\|(\Delta-\frac{2}{n}\tilde{S})^{-1}(S-\tilde{S})\|_{C^{0,\alpha}}\leq C\|S-\tilde{S}\|_{L^{p}}.$$
The rest of the proof remains unchanged.
\end{rem}
 
\section{Flow Approach}
Let $(M,J,g)$ be closed almost-Hermitian manifold of dimension $2n$. We denote by $\operatorname{vol}_g=\frac{F^n}{n!}$ the volume form, where $F$ is the fundamental form induced by $(J,g)$. In this section,
we study the Chern--Yamabe flow~\cite{Angella:2015aa}
\begin{equation}\label{genflow}
\left\{\begin{array}{ll}
\frac{\partial u}{\partial t}&=-\Delta u-s^C-g(\theta,du)+\lambda e^{\frac{2u}{n}},\\
\\
u(x,0)&=u_{0} \in L^{\infty}(M),
\end{array}
\right.
\end{equation}
where $\Delta$ is the Riemannian Laplacian with respect to $g$, $s^C$ is the Chern scalar curvature of $(J,g)$ and $\lambda$ is a constant.
\subsection{The case ${C(M,J,[g])}<0$} By Proposition~\ref{sign_fixed}, we can suppose without loss of generality that the function
$s^C<0$ everywhere on $M.$ We also recall the maximum principle for parabolic equations.
\begin{lem}\cite{MR2265040}
Let $F:\mathbb{R}\times [ 0,t_{0}] \to \mathbb{R}$ be a smooth function and $X(t)$ a smooth family of vector fields. Assume that $u\in C^{\infty}(M\times [0,t_{0}],\mathbb{R})$ satisfies
$$\frac{\partial u}{\partial t}\leq-\Delta u +g(X(t),\nabla u)+F(u,t),$$
and let $y(t)$ be the solution of 

$$\left\{\begin{array}{ll}
y'&=F(y,t),\\
\\
y(0)&=c_{0}.
\end{array}
\right.
$$

If $u(\cdot,0)\leq c_{0}$ then $u(x,t)\leq y(t)$ for all $t\in [0,t_{0}]$.
\end{lem}
Now, using the flow (\ref{genflow}), we recover Angella--Calamai--Spotti's result~\cite{Angella:2015aa} on the existence of negative constant Chern scalar
curvature metric in $[g]$ when the fundamental constant is negative. 
\begin{thm}\label{thm_1}
Suppose that ${C(M,J,[g])}<0$ and that $s^C<0$. Then for every $\lambda <0$ and $k \in \mathbb{N}$, the flow $(\ref{genflow})$, with $u_{0}\in L^{\infty}(M)$, converges to a solution of the equation
$$\Delta u +g(\theta,du)+s^C =\lambda e^{\frac{2u}{n}},$$
exponentially in the $C^{k}$ sense.
\end{thm}
\begin{proof}
The short time existence of the flow~(\ref{genflow}) is guaranteed by the classical parabolic equation theory~\cite{MR2265040,MR2744149}. Hence, there exists $T>0$ such that the solution exists for all $t\in [0,T)$. We first show that the solution is global, that is $T=+\infty$. We assume that,  $-b<s^C<-a<0$. Then,
$$\frac{\partial u}{\partial t}\leq -\Delta u+b+\lambda e^{\frac{2u}{n}}.$$
Using the maximum principle, we can compare it to the ODE
$$
\left\{ \begin{array}{ll}
\ddfrac{dy_{1}}{d t}&=b+\lambda e^{\frac{2y_{1}}{n}},\\
\\
y_{1}(0)&=c_{0}.
\end{array}
\right.
$$

The solution of this equation takes the form
$$y_{1}(t)=\frac{n}{2}\ln\left(-\frac{b}{\lambda}\frac{c_{1}e^{\frac{2bt}{n}}}{1+c_{1}e^{\frac{2bt}{n}}}\right),$$
where $$c_{1}=\frac{-\lambda e^{\frac{2c_{0}}{n}}}{b+\lambda e^{\frac{2c_{0}}{n}}}.$$
Similarly, we consider the ODE,
$$
\left\{ \begin{array}{ll}
\ddfrac{dy_{2}}{d t}&=-a-\lambda e^{\frac{-2y_{2}}{n}},\\
\\
y_{2}(0)&=c_{0}.
\end{array}
\right.
$$

Then again,
$$y_{2}(t)=-\frac{n}{2}\ln\left(-\frac{a}{\lambda}\frac{c_{2}e^{\frac{2at}{n}}}{c_{2}e^{\frac{2at}{n}}-1}\right),$$
where $$c_{2}=\frac{\lambda e^{\frac{-2c_{0}}{n}}}{a+\lambda e^{\frac{-2c_{0}}{n}}}.$$
If we choose now $c_{0}>>\|u_{0}\|_{L^{\infty}}$, then we have by the comparaison principle that
$$-y_{2}(t)\leq u(t) \leq y_{1}(t).$$
Since $y_{1}$ and $y_{2}$ are bounded, then $u$ does not blow up and exists for all time. Moreover $\|u(t)\|_{L^{\infty}}$ is bounded. The boundedness in $L^{\infty}(M)$ generates then a boundedness in $C^{2,\alpha}(M)$.\\
We set now $v=\frac{\partial u}{\partial t}$, then $v$ satisfies the linear parabolic equation
$$\frac{\partial v} {\partial t}=-\Delta v -g(\theta,dv)+\frac{2\lambda}{n}e^{\frac{2u}{n}}v.$$
Since $\lambda<0$, we have the existence of $\beta>0$ such that $$\|v\|_{L^{\infty}}\leq Ce^{-\beta t}.$$
But first, we need to fix the sign of $v$. So we consider $h=\frac{1}{2}v^{2}$, then
\begin{align}
\frac{\partial h}{\partial t}&=-v\Delta v -g(\theta,vdv)+\frac{2\lambda}{n}e^{\frac{2u}{n}}v^{2}\notag\\
&=-\Delta h +|\nabla v|^{2}-g(\theta,dh)+\frac{4\lambda}{n}e^{\frac{2u}{n}}h\notag\\
&=-\Delta h+g\left(\nabla h,\nabla \left(\ln(|v|)\right)\right)-g(\theta,dh)+\frac{4\lambda}{n}e^{\frac{2u}{n}}h\notag
\end{align}
Since $h\geq 0$, we can compare to the linear ODE
\[\begin{aligned}
y'&=\frac{4\lambda}{n}mh,\\
y(0)&=c_{0},
\end{aligned}\]
where $m=\displaystyle\min_{t\in \mathbb{R},x\in M} e^{\frac{2u(x,t)}{n}}$. So, if $c_{0}>\|h(0)\|_{L^{\infty}}$, we have that
$$h(t,x)\leq Ce^{\frac{4m\lambda}{n}}.$$
It follows then that
$$\|v(t)\|_{L^{\infty}}\leq Ce^{\frac{2m\lambda}{n}}.$$
To get a better convergence we consider the function $v_{1}=\partial_{i}v$. Then $v_{1}$ satisfies
\begin{align}
\frac{\partial v_{1}} {\partial t}&=-\Delta v _{1}-g(\theta,dv_{1})+\frac{2\lambda}{n}e^{\frac{2u}{n}}v_{1}-g(\partial_{i}\theta,dv)-\left(\partial_{i}g\right)(\theta,dv)+\frac{4\lambda}{n^{2}}\left(\partial_{i}u\right)e^{\frac{2u}{n}}v,\notag\\
&=-\Delta v _{1}-g(\theta,dv_{1})+\frac{2\lambda}{n}e^{\frac{2u}{n}}v_{1}+G(x,t).\notag
\end{align}
So $v_{1}$ satisfies the same equation as $v$ up to the term $G(x,t)$ which decays exponentially to zero. Using the same trick, we have that $\|v_{1}\|_{L^{\infty}}\leq Ce^{-\beta t}$.
This implies that $\frac{\partial u}{\partial t}$ converges to zero exponentially in $C^{1}(M)$. Hence, the flow converges to a solution of the desired equation in the $C^{1}$ sense. Iterating this argument yields a $C^{k}$ exponential convergence for for all $k>0$.
\end{proof}

\subsection{The case ${C(M,J,[g])}>0$.} We want now to investigate the case when ${C(M,J,[g])}$ is positive.
We restrict ourselves to a $2n$-dimensional closed balanced manifold $(M,J,g)$ (so the Lee form $\theta=0$) and
we will consider a slightly modified flow namely
\begin{equation}\label{flow2}
\left\{\begin{array}{ll}
\ddfrac{\partial u}{\partial t}&=-\Delta u -s^C+\lambda(t)\,e^{\frac{2u}{n}},\\
\\
u(x,0)&=0,
\end{array}
\right.
\end{equation}
where $\lambda(t)=\frac{\int_M s^C\operatorname{vol}_g}{\int_{M}e^{\frac{2u}{n}}\operatorname{vol}_g}$, $\Delta$ is the Riemannian Laplacian with respect to $g$
and $s^c$ the Chern scalar curvature of $(J,g).$

We denote by $H^{k,2}(M)$ the Sobolev space of functions on $M$ involving derivatives up to the order $k$.
We want to prove that if $s^C$ is small enough in $H^{k,2}$-norm (for $k>n$) then the flow~(\ref{flow2}) converges to a solution of the Chern--Yamabe problem.
The first property of the flow~(\ref{flow2}) is that $\int_{M}u\operatorname{vol}_g=0$ as long as the solution exists. Indeed, if we take $f(t)=\int_{M}u\operatorname{vol}_g$, then $f(0)=0.$ Moreover, we have that 
$$f'(t)=\int_{M}\frac{\partial{u}}{\partial t}\operatorname{vol}_g=\int_{M}-\Delta u-s^C + \lambda(t)e^{\frac{2u}{n}}\operatorname{vol}_g=0.$$
\begin{prop}
The solutions of $(\ref{flow2})$ exists globally.
\end{prop}
\begin{proof}
The short time existence is guaranteed by the classical parabolic PDE theory. Now, suppose that the solution $u$ exists on an interval $[0,T)$ where $T<+\infty$. Let $c>0$ to be fixed later and define $v=e^{-ct}u$. Then, $v$ satisfies
$$\frac{\partial v}{\partial t}=-\Delta v-cv -e^{-ct}s^C+\lambda(t)e^{\frac{2ve^{ct}}{n}-ct}.$$
Then, we have
$$\frac{\partial v}{\partial t}\leq-\Delta v-cv+A_{0}+\delta e^{Av},$$
where $\delta$ is an upper bound of $\lambda(t)$, $A_{0}$ an upper bound for $-s^{C}e^{-ct}$ for $t\in[0,T]$ and $A=\frac{2e^{cT}}{n}$. We compare then to the ODE
$$y'=\delta e^{Av}-cv+A_{0}.$$
By looking at the phase space of this autonomous ODE, we see that the equation has two equilibrium solutions if $c$ is big enough, let us call them $y_{0}<y_{1}$. So if $y(0)$ is chosen so that $0<y(0)<y_{1}$, we have that $y$ exists globally and bounded from above. But we have from the comparison principle that
$$u(t)\leq y(t)e^{ct}.$$
A similar bound also folds for $-u$.  This leads to a contradiction hence $T=+\infty$.
\end{proof}

\begin{lem}\label{lem_1}
Let $k>n$. Then, there exists $\delta>0$ and $\varepsilon_{0}>0$ such that if $\|s^C\|_{L^{2}}<\delta,$ we have $\|u\|_{H^{k,2}}<\varepsilon_{0}$ for all $t>0$.
\end{lem}
\begin{proof}
We stress first on the fact that $\int_{M}-s^C+\lambda(t)e^{\frac{2u}{n}}\operatorname{vol}_g=0$ along the flow. Now for $k$ big enough ($k>n$), the first eigenvalue of the operator $-\Delta$ on the space $H^{k}_{0}=\{u\in H^{k,2}(M)\,|\,\int_{M}u\operatorname{vol}_g=0\}$ is strictly negative hence $\|e^{-t\Delta}\|_{H^{k}_{0},H^{k}_{0}}\leq C_{0}e^{-ct}$. In fact, if we decompose the operator on an orthonormal basis of $L^{2}$ generated by the eigenfunctions of $-\Delta$, we see that we have a stronger result. That is, for $t>0$, there exists $C>0$, such that 
$\|e^{-t\Delta}\|_{H^{k}_{0},H^{0}_{0}}\leq C e^{-ct}$ with the convention that $H^{0,2}(M)=L^{2}(M)$.
 The solution $u(t)$ can be represented as
$$u(t)=\int_{0}^{t}e^{-(t-s)\Delta}\left(-s^C+\lambda(t)e^{\frac{2u}{n}}\right)dt.$$
So assume $\|s^C\|_{H^{k,2}}<\delta$, since $H^{k,2}(M)\rightharpoonup L^{\infty}(M)$ for $k>n$, we have that $$\|\lambda{t} e^{\frac{2u}{n}}\|_{L^{\infty}}\leq C\delta e^{C\|u\|_{H^{k,2}}}$$
Hence 
$$\|-s^C+\lambda(t)e^{\frac{2u}{n}}\|_{L^{2}}\leq C\delta e^{\|u\|_{H^{k,2}}}.$$
Let $t_{0}$ be the first time for which $\|u(t)\|_{H^{k,2}}=\varepsilon$. Since $-s^C+\lambda(t)e^{\frac{2u}{n}}\in H^{k}_{0}$, we have then $$\varepsilon\leq C \delta.$$
Thus if we take $\delta=O(\varepsilon^{2})$ we see that for $\varepsilon$ small enough $t_{0}$ cannot be reached. 
\end{proof}

\begin{proof}[Proof of Theorem \ref{thm_2}]
The proof of Theorem~\ref{thm_2} is in fact  a straightforward consequence of Lemma~\ref{lem_1} and decay estimates on the time derivative. So, we need to show the convergence of $v=\frac{\partial u}{\partial t}$ to zero in the $H^{k,2}$-norm as in the proof for the negative case. First, we observe that $\int_{M}v\operatorname{vol}_g=0$. Let $F(t)=\frac{1}{2}\|v\|_{L^{2}}^{2}$, then
\begin{align}
F'(t)&=-\|\nabla v\|_{L^{2}}^{2}+\lambda'(t)\int_{M}e^{\frac{2u}{n}}v\operatorname{vol}_g+\frac{2}{n}\lambda(t)\int_{M}v^{2}e^{\frac{2u}{n}}\operatorname{vol}_g\notag\\
&\leq -\|\nabla v\|_{L^{2}}^{2}+C\delta \|v\|_{L^{2}}^{2} +C\delta \|v\|_{L^{2}}^{2}\notag
\end{align}
Since $\int_{M}v\operatorname{vol}_g=0$ by the Poincar\'{e} inequality, we have that $\|v\|_{L^{2}}\leq C \|\nabla v\|_{L^{2}}$, therefore, for $\delta$ even smaller, we have that
$$F'(t)\leq -cF(t)$$
So $F$ decays exponentially. The rest of the proof follows then exactly like in the negative case.

\end{proof}
\section{Case of symmetric manifolds}
Let $(M,J,g)$ be a balanced closed manifold of dimension $2n$ and $s^C$ its Chern scalar curvature. In this section, we
assume the existence of a subgroup $G\subset Isom(M,g)$ such that $\dim G=2n-2$ and $G$ acts freely on $M$. Notice in this case that $\widetilde{M}=M/G$ has the structure of a $2$-dimensional manifold. 
We set $H^{1}_{G}(M)$ to be the set of $G$-invariant $H^{1,2}(M)$ functions, that is
$$H^{1}_{G}(M)=\{u\in H^{1,2}(M)\,|\, u(sx)=u(x), \forall x\in M, s\in G\}.$$
Notice that if $u\in H^{1}_{G}(M)$, then $u=v\circ \pi$ where $\pi:M\to \tilde{M}$ is the canonical projection and $v\in H^{1,2}(\widetilde{M})$. This method of equivariant reduction was heavily used for critical problems such as the classical Yamabe problem in~\cite{MR863646}, the CR--Yamabe problem in~\cite{MR2975684,MR3357578} and the spinorial Yamabe type problem~\cite{MR3490901}. This reduction process, usually provides a way to go from a critical setting to a subcritical setting. In our case, the problem is super-critical, so we will use this process to move from the super-critical setting to the critical setting which is very similar to the scalar curvature problem on Riemann surfaces.\\

We define the functional $\mathcal{E}$ on $H_G^{1}(M)$ by

$$\mathcal{E}(u)=\int_{M}\frac{1}{2}|\nabla u|^{2}+s^Cu\operatorname{vol}_g -\frac{n}{2}\int_M s^C\operatorname{vol}_g\ln\left(\int_M e^{\frac{2u}{n}} \operatorname{vol}_g\right)$$

If $s^C$ is invariant under the action of $G$, then the functional $\mathcal{E}$ is invariant under $G$ hence it descends to a functional on $\tilde{M}$. Notice that the critical points of $\mathcal{E}$ satisfy the Euler--Lagrange equation
$$\Delta u +s^C =\frac{\int_M s^C\operatorname{vol}_g}{\int e^{\frac{2u}{n}}\operatorname{vol}_g}e^{\frac{2u}{n}}.$$
\begin{prop}\label{prop_1}
Assume that $\int_M\operatorname{vol}_g=1.$ If $s^C$ is invariant under $G$ and $$\int_M s^C\operatorname{vol}_g={C(M,J,[g])} <2\pi n^{2}.$$ Then there exists a $G$-invariant metric $\tilde{g}=e^{\frac{2u}{n}}g$ of constant Chern scalar curvature.
\end{prop}

\begin{proof}
The proof is classical and similar to the 2-dimensional case for the Riemannian scalar curvature. Indeed, we recall first the Beckner's Inequality. Namely, if $(\Sigma,h)$ is a Riemann surface with a Riemannian metric $h$ and $u\in H^{1,2}(\Sigma)$ then there exists a constant $C_{\Sigma,h}$ such that
$$\ln\left(\int_{\Sigma}e^{u-\int_\Sigma u\operatorname{vol}_h}\operatorname{vol}_h\right)\leq \frac{1}{16\pi}\int_{\Sigma}|\nabla u|^{2}\operatorname{vol}_h +C_{\Sigma,h}.$$
By $G$-invariance, this previous inequality can be lifted to $H_{G}^{1}(M)$. Now, assume that $\int_M\operatorname{vol}_g=1$ then for every $c\in \mathbb{R}$ we have $$\mathcal{E}(u+c)=\mathcal{E}(u).$$
Thus withour loss of generality we can assume that $\int_M{u}\operatorname{vol}_g=0$. We have then
\begin{align}
\mathcal{E}(u)&\geq\left(\frac{1}{2}-\frac{\int_M s^C\operatorname{vol}_g}{4n^{2}\pi}\right) \int_{M}|\nabla u|^{2}\operatorname{vol}_g+\int_{M}s^Cu\operatorname{vol}_g -C,\notag\\
&\geq \left(\frac{1}{2}-\frac{\int_M s^C\operatorname{vol}_g}{4n^{2}\pi}-\varepsilon\right) \int_{M}|\nabla u|^{2}\operatorname{vol}_g-C,\notag
\end{align}
for some constant $C$. So if $\int_M s^C\operatorname{vol}_g<2\pi n^{2}$, the functional is coercive and the minimization process follows to yield a minimum to the functional.
\end{proof}

\bibliographystyle{abbrv}

\bibliography{bibliography}

\end{document}